\documentclass[11pt]{amsart}
\usepackage[margin=2.5cm]{geometry}
\numberwithin{equation}{section}
\usepackage{amssymb}
\usepackage{amsmath, amsfonts,amsthm,amssymb,amscd, verbatim,graphicx,color,multirow,booktabs, caption,tikz,tikz-cd, mathdots,bm}
\usepackage{tikz-cd}
 \usepackage{hyperref}
\usetikzlibrary{positioning}

\newtheorem{theorem}{Theorem}[section]
\newtheorem{lemma}[theorem]{Lemma}

\newtheorem{proposition}[theorem]{Proposition}
 \newtheorem{corollary}[theorem]{Corollary}

      \theoremstyle{definition}
     \newtheorem*{definition}{Definition}
     \newtheorem{example}[theorem]{Example}
     
     \theoremstyle{remark}
     \newtheorem{remark}[theorem]{Remark}

\newcommand{\Sym}{\mathop{\mathrm{Sym}}}
\newcommand{\Alt}{\mathop{\mathrm{Alt}}}

 \definecolor{mycolor}{rgb}{0.55,0.0,0.16}
  \definecolor{myred}{rgb}{0.6,0.0,0.16}
  \definecolor{mygreen}{rgb}{0.0,0.6,0.16}
  \definecolor{myviolet}{rgb}{1,0,1}

\hypersetup{
colorlinks=true,
linkcolor=true,
linktocpage=true,
pageanchor=true,
hyperindex=true
}

 \AtBeginDocument{
     \hypersetup{
  linkcolor=myred,
  urlcolor=mycolor,
citecolor=mygreen
}
     }

\makeatletter
\@namedef{subjclassname@2020}{%
  \textup{2020} Mathematics Subject Classification}
\makeatother

\begin{document}
\title[Products of subgroups, subnormality, and relative orders of elements]{Products of subgroups, subnormality,\\and relative orders of elements}
 \author[L. Sabatini]{Luca Sabatini}
\address{Luca Sabatini, Alfr\'ed R\'enyi Institute of Mathematics, Hungarian Academy of Sciences\newline
Re\'altanoda utca 13-15, H-1053, Budapest, Hungary} 

\email{sabatini@renyi.hu, sabatini.math@gmail.com}
\subjclass[2020]{Primary 20D40, 20D25, 20F99.}
 \keywords{Relative order; product of subgroups; subnormal subgroup.}        
	\maketitle
	
	%\begin{center}
	%Luca Sabatini\\
	%Alfr\'ed R\'enyi Institute of Mathematics, Hungarian Academy of Sciences\\
%Re\'altanoda utca 13-15, H-1053, Budapest, Hungary
	%\end{center}

        \begin{abstract}
        Let $G$ be a group.
        We give an explicit description of the set of elements $x \in G$ such that $x^{|G:H|} \in H$ for every subgroup of finite index $H \leqslant G$.
       This is related to the following problem: given two subgroups $H$ and $K$, with $H$ of finite index,
        when does $|HK:H|$ divide $|G:H|$?
          \end{abstract}
          
          \vspace{0.2cm}
	\begin{section}{Introduction}
	
	Let $G$ be an arbitrary group, and let us write $H \leqslant_f G$ to say that $H$ is a subgroup of $G$ of finite index.
	Let $x \in G$ and $H \leqslant_f G$.
	If $H$ is a normal subgroup of $G$, then it is easy to see that $x^{|G:H|} \in H$.
	The same is not true in general:
	fixed $H \leqslant_f G$, the set $\{ x \in G: x^{|G:H|} \in H \}$ may not even be closed under multiplication
	(take $G=\Sym(3)$ and $H=\langle (1 \> 2) \rangle$).
	The goal of this paper is to understand this phenomenom and its implications.
	As far as we can see, this has not been dealt with before in the literature.
	
	\begin{definition}
	Let $x \in G$ and $H \leqslant G$.
	The {\itshape relative order} of $x$ with respect to $H$ is
	$$ o_H(x) \> := \> | \langle x \rangle : \langle x \rangle \cap H | . $$ 
	\end{definition}
	\vspace{0.1cm}
	
	The following result is proved in Section \ref{sectPrelim}.

	\begin{lemma} \label{lemRelOrd}
	 Let $n \geq 1$.
	Then $x^n \in H$ if and only if $o_H(x)$ is finite and divides $n$.
	\end{lemma}
	 \vspace{0.1cm}
	
	Given $H,K \leqslant G$,
	 $|HK:H|$ is the cardinality of the set of all cosets of $H$ which are intersected by $K$
	 (we refer to Section \ref{sectPrelim} for more details).
	Since $o_H(x)=|H \langle x \rangle :H|$, we obtain
	 
	\begin{corollary} \label{corCrit}
	$x^{|G:H|} \in H$ if and only if $|H \langle x \rangle:H|$ divides $|G:H|$.
	\end{corollary}
	\vspace{0.1cm}
	
	If $H,K \leqslant_f G$, then $|HK:H|$ divides $|G:H|$ if and only if $|HK:K|$ divides $|G:K|$.
	If $G$ is finite, both are equivalent to $|HK|$ dividing $|G|$.
	In Section \ref{sectPS}, we prove the following two results:
	
	\begin{proposition} \label{propSubn}
	Let $H \lhd \lhd \> G$. Then $|HK:K|$ divides $|G:K|$ for every $K \leqslant_f G$.
	\end{proposition}
	% \vspace{0.1cm}
	
	\begin{theorem} \label{thSubn2}
	Let $H \leqslant_f G$. Then $H \lhd \lhd \> G$ if and only if $|HK:H|$ divides $|G:H|$ for every $K \leqslant G$.
	\end{theorem}
	\vspace{0.1cm}
	
	The converse of Proposition \ref{propSubn} is not true in general (see Example \ref{exBad}).
	In particular, some attention is needed with subgroups of infinite index.
	During the preparation of this manuscript,
	the author found out that the finite version of Theorem \ref{thSubn2} already appeared in \cite[Theorem 2]{2022Levy}.
	In Section \ref{sectExpSub}, we study the following class of subgroups
	
	\begin{definition}
	A subgroup $H \leqslant_f G$ is {\itshape exponential} if $x^{|G:H|} \in H$ for every $x \in G$.
	\end{definition}
	\vspace{0.1cm}
	
	This is a generalization of subnormality, and we prove that it is equivalent to normality in some cases,
	namely for the Hall subgroups of a finite group and for the maximal subgroups of a solvable group.
	From the dual point of view, in Section \ref{sectSG} we study the set
	$$ S(G) \> := \> \{ x \in G : x^{|G:H|} \in H \mbox{ for every } H \leqslant_f G \} . $$
	At first glance $S(G)$ is quite elusive, and indeed working directly with the definition is not easy.
	Using the results of Section \ref{sectPS}, we give an elementary proof of the next theorem.
	Given $N \lhd G$, let $F_N(G)$ be the preimage of $F(G/N)$,
	where $F(G)$ denotes the Fitting subgroup of $G$.
	
	\begin{theorem} \label{thSExp}
	If $G$ is any group, then $S(G) = \cap_{N \lhd_f G} F_N(G)$.
	\end{theorem}
	\vspace{0.1cm}
	
	In particular, $S(G)=F(G)$ when $G$ is finite (Proposition \ref{propFitting}).
	Of course, Theorem \ref{thSExp} implies that $S(G)$ is closed under multiplication,
	a fact which is not immediately clear from the definition.
	
	\end{section}

	\vspace{0.2cm}
	\begin{section}{Preliminaries} \label{sectPrelim}
	
	We start with the proof of the key Lemma \ref{lemRelOrd}.
	
	\begin{proof}[Proof of Lemma \ref{lemRelOrd}]
	Let $ord_H(x) := \min \{ n \geq 1: x^n \in H \}$.
	We first notice that $o_H(x)=ord_H(x)$.
	Indeed, from the definitions we have $o_H(x)=o_{H \cap \langle x \rangle}(x)$ and $ord_H(x)=ord_{H \cap \langle x \rangle}(x)$.
	The fact that $o_{H \cap \langle x \rangle}(x) = ord_{H \cap \langle x \rangle}(x)$
	is a simple exercise.
	Now, the ``if'' part of the statement is trivial.
	On the other hand, if $x^n \in H$ for some $n \geq 1$, then clearly $ord_H(x) < \infty$. 
	Let $n = q \cdot ord_H(x) + r$ with $r,q \geq 0$ and $r<ord_H(x)$. 
	Since $H$ is a subgroup, the fact that $x^n = x^{q \cdot ord_H(x)} x^r \in H$ implies that $x^r \in H$,
	which in turn means $r=0$.
	\end{proof}
	\vspace{0.1cm}
	
	The bulk of this paper is about finite groups.
	We summarize here the basic tools and notation that are used with regard to general non-finite groups.
	Let $G$ be an arbitrary group and $H,K \leqslant G$.
	If $H$ and $K$ have finite index, then so has $H \cap K$, and $|G:H \cap K|=|G:H||H:H \cap K|$.
	As we have said in the introduction, we write $|HK:H|$ for the cardinality of the set of all cosets of $H$ which are intersected by $K$.
	This is not accidental, because the product set $HK = \{ hk : h \in H,k \in K \}$ is a union of cosets of $H$.
	It is not relevant to distinguish between left-cosets and right-cosets,
	since $k \in Hx$ if and only if $k^{-1} \in x^{-1}H$.
	 We also observe that $|HK:H|=|K:H \cap K|=|KH:H|$.\\
	 The {\itshape finite residual} $R(G)$ is the intersection of the subgroups of $G$ of finite index.
	 If $R(G)=1$, then $G$ is said to be {\itshape residually finite}. It is easy to check that $G/R(G)$ is always residually finite.
	Finally, the {\itshape Fitting subgroup} $F(G)$ is defined as the subgroup generated by the nilpotent normal subgroups,
	and coincides with the set of the elements $x \in G$ such that the normal closure $\langle x \rangle^G$ is nilpotent \cite{2008Cas}.
	In general, this is a stronger condition than $\langle x \rangle$ being subnormal in $G$.
	If $G$ is finite, then $F(G)$ itself is nilpotent, i.e. it is the largest nilpotent normal subgroup.
	 
	\end{section}

	\vspace{0.2cm}
	\begin{section}{Products of subgroups} \label{sectPS}
	
	The proof of Proposition \ref{propSubn} follows immediately from the following
	
	\begin{lemma} \label{lemTool}
	Let $H \lhd M \leqslant G$, and let $K \leqslant_f G$. Then $|HK:K|$ divides $|MK:K|$.
	\end{lemma}
	\begin{proof}
	We have to prove that the ratio
	$$ \frac{|MK:K|}{|HK:K|} = \frac{|M:M \cap K|}{|H:H \cap K|} $$ 
	is an integer.
	Now $H \lhd M$ implies that $H(M \cap K)$ is a subgroup of $M$, and so we can write
	  \begin{align*}
|M:M \cap K|  & \> = \> 
|M:H(M \cap K)| |H(M \cap K):M \cap K| \\ & \> = \> 
 |M:H(M \cap K)| |H:H \cap K|  .
\end{align*}
	In particular, the original ratio equals $|M:H(M \cap K)|$.
	\end{proof}
	\vspace{0.1cm}
	
	We continue with the easiest direction of Theorem \ref{thSubn2}.
	
	\begin{lemma} \label{lemTool2}
	Let $H \leqslant_f M \leqslant_f G$, and let $K \leqslant G$. Then $\frac{|HK:H|}{|MK:M|}=|M \cap K:H \cap K|$.
	\end{lemma}
	\begin{proof}
	We have
	\begin{align*}
	\frac{|HK:H|}{|MK:M|}  & \> = \> 
	\frac{|K:H \cap K|}{|K:M \cap K|}  \\ & \> = \> 
	\frac{|K:M \cap K| |M \cap K:H \cap K|}{|K:M \cap K|}  \\ & \> = \> 
	 |M \cap K:H \cap K| . \qedhere
	\end{align*}
	\end{proof}
	\vspace{0.1cm}
	
	We prove the claim of Theorem \ref{thSubn2} by induction on the subnormal defect of $H$, so
	let $H \lhd_f M \lhd \lhd_f \> G$, and $K \leqslant G$. Using Lemma \ref{lemTool2}, we have
	$$ \frac{|G:H|}{|HK:H|}  = \frac{|G:M||M:H|}{|MK:M| |M \cap K:H \cap K|} . $$
	By induction, it is sufficient to prove that $\frac{|M:H|}{|M \cap K:H \cap K|}$ is an integer.
	Now $H \lhd M$ implies that $H(M \cap K)$ is a subgroup of $M$, and so we can write
	\begin{align*}
	|M:H|  & \> = \> 
	|M:H(M \cap K)| |H(M \cap K):H|  \\ & \> = \> 
	 |M:H(M \cap K)||M \cap K:H \cap K| .
	\end{align*}
	This concludes the proof of the ``only if'' part.

	\begin{subsection}{The Kegel-Wielandt-Kleidman theorem, revisited}
	
	\begin{definition}
	Let $G$ be a finite group, $H \leqslant G$, and let $p$ be a prime.
	Then $H$ is {\itshape $p$-subnormal} in $G$ if $H \cap P$ is a $p$-Sylow of $H$ for every $p$-Sylow $P$ of $G$. 
	\end{definition}
	\vspace{0.1cm}
	
	We characterize $p$-subnormality with the following
	
	\begin{lemma} \label{lemPSubNChar}
	A subgroup $H$ is $p$-subnormal if and only if $|HP|$ divides $|G|$ for every $p$-Sylow $P \leqslant G$.
	\end{lemma}
	\begin{proof}
	We have that $H \cap P$ is a $p$-Sylow of $H$ if and only if $|H:H \cap P|=|HP:P|$ is not divisible by $p$.
	Since $|H:H \cap P|$ is a divisor of $|G|$, the last condition is equivalent to $|HP:P|$ dividing $|G:P|$, i.e. $|HP| \mid |G|$.
	\end{proof}
	\vspace{0.1cm}
	
	The famous Kegel-Wielandt conjecture \cite{1962Kegel,1980W},
	proved by Kleidman \cite{1991Kleidman} using the classification of the finite simple groups,
	says that $H \lhd \lhd \> G$ whenever $H$ is $p$-subnormal for every $p$.
	
	\begin{theorem}[Kegel-Wielandt conjecture] \label{thKW}
	If $|HP|$ divides $|G|$ for every Sylow subgroup $P \leqslant G$, then $H \lhd \lhd \> G$.
	\end{theorem}
	\vspace{0.1cm}
	
	See \cite{1993GKL} for some consequences of $p$-subnormality for a single $p$.
	The ``if'' part of Theorem \ref{thSubn2} follows easily.
	Let $H \leqslant_f G$, and assume that $|HK:H|$ divides $|G:H|$ for every $K \leqslant G$.
	Let $N \lhd_f G$ be the normal core of $H$, and let $N \leqslant K \leqslant G$ be any intermediate subgroup.
	Working with $G/N$ and $K/N$, Theorem \ref{thKW} gives $H/N \lhd \lhd \> G/N$, i.e. $H \lhd \lhd \> G$.\\
	
	We point out that Kegel \cite{1962Kegel} did not use the classification to prove Theorem \ref{thKW} when $H$ is solvable.
	We give a very short proof in the case where $H$ is nilpotent,
	which is enough for the characterization of $S(G)$ we will present in Section \ref{sectSG}.
	
	\begin{lemma}[Kegel-Wielandt for nilpotent subgroups] \label{lemKleidNilp}
	Let $H \leqslant G$ be a nilpotent subgroup of the finite group $G$.
	If $|HP|$ divides $|G|$ for every Sylow subgroup $P \leqslant G$, then $H \lhd \lhd \> G$.
	\end{lemma}
	\begin{proof}
	Suppose that $H$ is not subnormal, and in particular $H \nleqslant F(G)$.
	So there exists a $p$-element $x$ such that $x \in H \setminus F(G)$.
	Since $x \notin O_p(G)$, there exists a $p$-Sylow $P$ of $G$ such that $x \notin P$.
	By hypothesis $H \cap P$ is a $p$-Sylow of $H$ and, since $H$ is nilpotent, $H \cap P$ contains all the $p$-elements of $H$.
	This contradicts the fact that $x \notin P$.
	\end{proof}
	\vspace{0.1cm}
	
	Levy \cite{2022Levy} proves the same result when $H$ is a $p$-subgroup of $G$.
	Another consequence of Theorem \ref{thSubn2} is that $p$-subnormality for every $p$ implies that $|HK|$ divides $|G|$ for every $K \leqslant G$.
	We provide an elementary proof of this fact.
	
	\begin{lemma}
	Let $G$ be a finite group and $H \leqslant G$.
	If $|HP|$ divides $|G|$ for every Sylow $P \leqslant G$,
	then $|HK|$ divides $|G|$ for every $K \leqslant G$.
	\end{lemma}
	\begin{proof}
	Let $K \leqslant G$. We have to show that $|HK:K|=|H:H \cap K|$ divides $|G:K|$.
	Let $p^\alpha$ be a prime power that divides $|H:H \cap K|$.
	Since $p^\alpha$ is arbitrary, it is sufficient to prove that $p^\alpha \mid |G:K|$.
	Let $P_0 \leqslant K$ be a $p$-Sylow of $K$, and let $P \leqslant G$ be a $p$-Sylow of $G$ such that $P \cap K=P_0$.
	Of course, $p^\alpha \mid |H:H \cap P_0|$.
	By hypothesis $|H:H \cap P|=|HP : P|$ divides $|G:P|$, and so is not divisible by $p$.
	Therefore, $p^\alpha \mid |H \cap P:H \cap P_0|$.
	Now $|H \cap P:H \cap P_0|=|(H \cap P)P_0:P_0|$,
	and this divides $|P:P_0|$ because $P$ is a $p$-group.
	So $p^\alpha \mid |P:P_0|$, and then of course $p^\alpha \mid |G:P_0|$.
	Since $p \nmid |K:P_0|$, we obtain $p^\alpha \mid |G:K|$ as desired.
	\end{proof}
	\vspace{0.1cm}
	
	\end{subsection}
	
	\end{section}

	\vspace{0.2cm}
	\begin{section}{Exponential subgroups} \label{sectExpSub}
	
	We write $H \leqslant_{exp} G$ if $x^{|G:H|} \in H$ for all $x \in G$.
	We observe immediately that exponentiality is preserved by quotients.
	
	\begin{lemma} \label{lemExpQuot}
 Let $N \lhd G$, and $N \leqslant H \leqslant G$. Then $H \leqslant_{exp} G$ if and only if $H/N \leqslant_{exp} G/N$.
 \end{lemma}
 \begin{proof}
 Let $x \in G$ and $H \leqslant_{exp} G$. Then $(Nx)^{|G/N:H/N|}=Nx^{|G:H|} \in H/N$ and so $H/N \leqslant_{exp} G/N$.
 If $H/N \leqslant_{exp} G/N$, then $Nx^{|G:H|} = (Nx)^{|G/N:H/N|} \in H$, and so $x^{|G:H|} \in H$.
 \end{proof}
 \vspace{0.1cm}
	
	Since exponential subgroups have finite index,
	we can apply Lemma \ref{lemExpQuot} with the normal core, and work with a finite group.
	Let $G$ be a finite group and $H \leqslant G$. From Corollary \ref{corCrit} and Theorem \ref{thSubn2}, we have
	
	\begin{itemize}
	\item $H \lhd \lhd \> G$ if and only if $|HK|$ divides $|G|$ for every $K \leqslant G$;
	\item $H \leqslant_{exp} G$ if and only if $|HC|$ divides $|G|$ for every cyclic $C \leqslant G$.
	\end{itemize}
	
	We stress that $H \leqslant_{exp} G$ whenever $|G:H|$ is a multiple of the exponent $\exp(G)$.
	
	\begin{remark}
	Every finite group of order other than a prime has a non-trivial exponential subgroup:
	if $\exp(G)<|G|$, then it is sufficient to take any subgroup whose order divides $|G|/\exp(G)$.
	Otherwise, all the Sylow subgroups of $G$ are cyclic,
	and it is well known that $G$ is solvable.
	In particular, $G$ has a non-trivial normal subgroup, which is certainly exponential.\\
	We notice a difference with the stronger condition that $HK$ is a subgroup for every $K$
	i.e. $H$ is a {\itshape permutable} subgroup.
	Indeed, it is easy to prove that if $HC$ is a subgroup for every cyclic $C \leqslant G$, 
	 then $HK$ is a subgroup for every $K \leqslant G$.
	\end{remark}
	 \vspace{0.1cm}
	
	 For every $n \geq 1$, let $G^n := \langle \{ x^n : x \in G \} \rangle$.
	The exponential subgroups of $G$ of index $n$ are in correspondence with the subgroups of $G/G^n$ of index $n$.
	Since $G^n$ is characteristic, the property of being exponential is preserved by automorphisms.
	Moreover, we have the following
	
	\begin{lemma} \label{lemNew}
	Let $H \leqslant G$ have a trivial characteristic core.
	Then $H \leqslant_{exp} G$ if and only if $|G:H|$ is a multiple of the exponent of $G$.
	\end{lemma}
	\begin{proof}
	Let $n=|G:H|$.
	By the exponentiality of $H$ we have $G^n \leqslant H$.
	Since $G^n$ is a characteristic subgroup of $G$ contained in $H$, we obtain $G^n=1$.
	But this means exactly that $n$ is a multiple of $\exp(G)$.
	The converse is trivial.
	\end{proof}
	\vspace{0.1cm}
	
	In general, there exist non-subnormal exponential subgroups whose index is not a multiple of the exponent.
	 A simple example is $G=C_4 \times \Sym(3)$ and $H \cong C_2 \times C_2$.
	 The following corollaries of Lemma \ref{lemNew} are obtained with the same strategy.
	
\begin{corollary}
	Let $H \leqslant G$ be a Hall subgroup. If $H \leqslant_{exp} G$, then $H \lhd G$.
	\end{corollary}
	\begin{proof}
	Suppose that $H$ is not normal, and let $N \lhd G$ be the normal core of $H$.
	Since $H/N$ is a Hall subgroup of $G/N$, by induction and Lemma \ref{lemExpQuot}, we can assume that $H$ is core-free.
	Now $\exp(G)$ captures every prime dividing $|G|$, and so the contradiction is given by Lemma \ref{lemNew}.
	\end{proof}
	% \vspace{0.1cm}
	
	\begin{corollary} \label{corMaxExp}
	Let $M \leqslant G$ be a maximal subgroup of the solvable group $G$. If $M \leqslant_{exp} G$, then $M \lhd G$.
	\end{corollary}
	\begin{proof}
	Suppose that $M$ is not normal, and let $N \lhd G$ be the normal core of $M$.
	Since $M/N$ is a maximal subgroup of $G/N$, by induction and Lemma \ref{lemExpQuot}, we can assume that $M$ is core-free.
	Now $|G:M|=q^\alpha$ for some prime power $q^\alpha$.
	If $G$ is a $q$-group we are done. Otherwise, the contradiction is given by Lemma \ref{lemNew}.
	\end{proof}
	\vspace{0.1cm}
	
	We cannot drop the hypothesis of solvability in Corollary \ref{corMaxExp}:
	the alternating group $G=\Alt(10)$ has a conjugacy class of maximal subgroups $M$ of size $720$.
	 Since $\exp(G)=2520=|G:M|$, it appears that $M$ is an exponential maximal subgroup which is not normal.\\
	We conclude this section with the hereditary properties of exponential subgroups.

\begin{lemma}
The following are true:
\begin{itemize}
\item If $H \leqslant_{exp} M \leqslant_{exp} G$, then $H \leqslant_{exp} G$;
\item The intersection of exponential subgroups is exponential.
\end{itemize}
\end{lemma}
\begin{proof}
Let $x \in G$. Since $M \leqslant_{exp} G$, we have $m=x^{|G:M|} \in M$.
Then $x^{|G:H|}=m^{|M:H|} \in H$.
To prove the second statement, it is sufficient to notice that $|G:H \cap K|$ is a multiple of both $|G:H|$ and $|G:K|$.
\end{proof}
\vspace{0.1cm}

	Other important properties of the lattice of the subnormal subgroups are not true for exponential subgroups,
	and the dihedral group $G=D_{12}$ is a good source of counterexamples.
	Every subgroup of $G$ whose order is $2$ is exponential in $G$, since $\exp(G) = 6$.
Let $H$ be any non-central subgroup of order $2$. Now
	
	\begin{itemize}
	\item The subgroup $H_1 = \langle H,Z(G) \rangle \cong C_2 \times C_2$
 provides a counterexample to the statement that two exponential subgroups generate an exponential subgroup:
 choosing any involution $x \in  G \setminus H_1$ we get $x^{|G:H_1|} = x \notin H_1$.
 \item The subgroup $H_2$ which satisfies $H < H_2 \cong \Sym(3)$ provides a counterexample to the statement that
the intersection of an exponential subgroup of $G$ with any subgroup of $G$ is exponential in that subgroup:
 choosing any involution $x \in H_2 \setminus H$, we get that $H$ is not exponential in $H_2$ although it is exponential in $G$.
	\end{itemize}
	
	\end{section}

	\vspace{0.2cm}
	\begin{section}{The set $S(G)$} \label{sectSG}
	
	Let us recall the definition of $S(G)$ given in the introduction: 
	$$ S(G) \> := \> \{ x \in G : x^{|G:H|} \in H \mbox{ for every } H \leqslant_f G \} . $$
	From Corollary \ref{corCrit}, we have
	$$ S(G) \> = \> \{ x \in G : |H \langle x \rangle:H| \mbox{ divides } |G:H| \mbox{ for every } H \leqslant_f G \} . $$
	The results of Section \ref{sectPS} allow to settle the finite case easily:
	 
	\begin{proposition} \label{propFitting}
	If $G$ is finite, then $S(G)=F(G)$.
	\end{proposition}
	\begin{proof}
	Let $x \in G$. Then $x \in S(G)$ if and only if $|H \langle x \rangle|$ divides $|G|$ for every $H \leqslant G$.
	From Proposition \ref{propSubn} and Lemma \ref{lemKleidNilp}, this is equivalent to $\langle x \rangle \lhd \lhd \> G$, i.e. $x \in F(G)$.
	\end{proof}
	\vspace{0.1cm}

	\begin{subsection}{A top-down approach}
	
	Let $G$ be an arbitrary group and let $R(G) = \cap_{H \leqslant_f G} H$ be its finite residual.
	The condition in the definition of $S(G)$ is empty on $R(G)$, and so $R(G) \subseteq S(G)$.
	In fact, $S(G)$ is the preimage of $S(G/R(G))$ under the projection $G \twoheadrightarrow G/R(G)$. 
	
	 \begin{lemma} \label{lemReductionRF}
	Let $N \lhd G$.
	Then $S(G/N) = \{ Nx : x^{|G:H|} \in H \mbox{ for every } N \leqslant H \leqslant_f G \}$.
	In particular, $S(G/R(G)) = S(G)/R(G)$.
	\end{lemma}
	\begin{proof}
	Let $x \in G$ and $N \leqslant H \leqslant_f G$.
	The equality $(Nx)^{|G:H|} = Nx^{|G:H|}$ implies that $Nx \in H/N$ if and only if $x^{|G:H|} \in H$,
	and the first part follows because $H$ is arbitrary.
	The second part follows because $R(G)$ contains all the subgroups of $G$ of finite index.
	\end{proof}
	\vspace{0.1cm}
	
	As a consequence of Lemma \ref{lemReductionRF}, we can assume that $G$ is residually finite.
	Given $N \lhd G$, let $F_N(G)$ be the preimage of $F(G/N)$.
	
	\begin{proof}[Proof of Theorem \ref{thSExp}]
	We have to prove that $S(G) = \cap_{N \lhd_f G} F_N(G)$.
	Let $x \in S(G)$ and $N \lhd_f G$.
	From Lemma \ref{lemReductionRF} and Proposition \ref{propFitting} we have $Nx \in S(G/N)=F(G/N)$, i.e. $x \in F_N(G)$.\\
	On the other hand, let $x \in \cap_{N \lhd_f G} F_N(G)$ and $H \leqslant_f G$.
	If $N \lhd_f G$ is the normal core of $H$, then in particular $x \in F_N(G)$.
	From Proposition \ref{propFitting} we have
	$$ Nx \in \frac{F_N(G)}{N} = F(G/N) = S(G/N) , $$
	and so Lemma \ref{lemReductionRF} provides $x^{|G:H|} \in H$.
	The proof follows because $H$ is arbitrary.
	\end{proof}
	\vspace{0.1cm}
	
	The following observation deletes a bunch of terms from $\cap_{N \lhd_f G} F_N(G)$.
	
	\begin{lemma}
	Let $G$ be a finite group and $N \lhd G$. Then $F(G) \leqslant F_N(G)$.
	\end{lemma}
	\begin{proof}
	We have that $NF(G)/N \cong F(G)/(N \cap F(G))$ is a nilpotent normal subgroup of $G/N$.
	Then $NF(G)/N \leqslant F(G/N) = F_N(G)/N$, and so $NF(G) \leqslant F_N(G)$.
	\end{proof}
	% \vspace{0.1cm}
	
	\begin{corollary}
	If $N,K \lhd_f G$ and $K \leqslant N$, then $F_K(G) \leqslant F_N(G)$.
	\end{corollary}
	\vspace{0.1cm}
	
	As a particular case of Theorem \ref{thSExp}, we have
	
	\begin{proposition} \label{propEquiv}
	Let $G$ be a group. The following are equivalent:
	
	\begin{itemize}
	\item[(A)] $G=S(G)$;
	\item[(B)] every subgroup of finite index of $G$ is exponential;
	\item[(C)] every finite quotient of $G$ is nilpotent;
	\item[(D)] every subgroup of finite index of $G$ is subnormal.
	\end{itemize}
	
	\end{proposition}
		\begin{proof}
		This follows easily from Theorem \ref{thSExp}.
		\end{proof}
	 \vspace{0.1cm}
	
	We say that a group $G$ is $S$-free if $S(G)=1$.

\begin{lemma}
Let $G$ be a group which is residually $S$-free. Then $S(G)=1$.
\end{lemma}
\begin{proof}
Let $1 \neq x \in G$. By definition, there exists $N \lhd G$ such that $x \notin N$ and $S(G/N)=1$. In particular $Nx \notin S(G/N)$,
and so from Lemma \ref{lemReductionRF} we obtain $x \notin S(G)$. Since $x$ is arbitrary, it follows that $S(G)=1$.
\end{proof} 
%	\vspace{0.1cm}
	
	\begin{corollary}
	If $F$ is a finitely generated free group, then $S(F)=1$.
	\end{corollary}
	\vspace{0.1cm}
	
	\end{subsection}

	\begin{subsection}{Baer groups and $S$-groups}
	
	Following a different approach, now we study $S(G)$ starting from the subgroups of $G$.
	This will provide a counterexample to the converse of Proposition \ref{propSubn}.
	
	\begin{comment}
	\begin{lemma} \label{lemSetSub}
	If $N \lhd \lhd \> G$, then $S(N) \subseteq S(G)$.
	\end{lemma}
	\begin{proof}
	First suppose $N \lhd G$, $x \in S(N)$ and $H \leqslant_f G$.
	Then $N \cap H \leqslant_f N$ and by definition we have $x^{|N:N \cap H|} \in H$.
	Now $|N:N \cap H|=|NH:H|$ is a divisor of $|G:H|$, and so $x^{|G:H|} \in H$.
	Thus $x \in S(G)$ because $H$ is arbitrary.
	The proof for $N \lhd \lhd \> G$ is obtained by induction on the subnormal defect.
	\end{proof}
	% \vspace{0.1cm}
	
	\begin{lemma}
	Let $G$ be a residually finite group. If $x \in S(G)$, then $\langle x \rangle^G$ is residually nilpotent.
	\end{lemma}
	\begin{proof}
	Let $x \in S(G)$, $K = \langle x \rangle^G$, and $1 \neq g \in K$.
	We have to prove that there exists $J \lhd K$ such that $g \notin J$ and $K/J$ is nilpotent.
	Since $G$ is residually finite, there exists $N \lhd_f G$ such that $g \notin N$.
	Let $J = K \cap N \lhd K$. Then $K/J \cong NK/N$.
	By Lemma \ref{lemReductionRF} and Proposition \ref{propFitting}, we have $Nx \in S(G/N)=F(G/N)$.
	So $NK/N \leqslant F(G/N)$ is nilpotent, and the proof is complete.
	\end{proof}
	\vspace{0.1cm}
	\end{comment}
	
		Let $B(G):=\{ x \in G : \langle x \rangle \lhd \lhd \> G \}$ be the {\itshape Baer radical} of $G$.
	It is clear that $B(G)$ is a characteristic subgroup.
	Moreover, $B(G)$ coincides with $F(G)$ if $G$ is finite, but it can be much larger in general (see \cite[Example 85]{2008Cas}).
	A group which equals its Baer radical is called a {\itshape Baer group}.
	The same argument in the proof of Proposition \ref{propFitting} shows that $B(G) \subseteq S(G)$.
	We say that a group is an {\itshape $S$-group} if it satisfies the equivalent conditions of Proposition \ref{propEquiv}.
	It is easy to see that the class of $S$-groups is closed by subgroups of finite index and quotients.
	Of course, every Baer group is an $S$-group.
	
	\begin{comment}
	\begin{remark}
	From Lemma \ref{lemReductionRF} we have $S(G)/B(G) \subseteq S(G/B(G))$.
	On the other hand, $S(G)/B(G)$ and $B(G/B(G))$ are not comparable in general:
	if $G$ is a non-nilpotent finite solvable group, then $S(G)/B(G)$ is trivial while $B(G/B(G))$ is not.
	On the other hand, if $1=B(G)<S(G)$, then $B(G/B(G))$ is trivial while $S(G)/B(G)$ is not.
	\end{remark}
	\vspace{0.1cm}
	\end{comment}
	
	\begin{proposition}[Theorem 73 in \cite{2008Cas}] \label{propBaer}
	A group is a Baer group if and only if every its finitely generated subgroup is subnormal and nilpotent.
	In particular, every finitely generated Baer group is nilpotent.
	\end{proposition}
	 \vspace{0.1cm}
	
	By Propositions \ref{propEquiv} and \ref{propBaer}, every finitely generated non-nilpotent $p$-group is an $S$-group which is not Baer.
	The next theorem of Wilson \cite{1971Wilson} provides many groups with {\itshape trivial} Baer radical.
	We recall that an infinite group is just-infinite if every its proper quotient is finite.
	
	\begin{theorem}[Theorem 2 in \cite{1971Wilson}] \label{thWilson}
	Let $G$ be a just-infinite group.
	If $B(G) \neq 1$, then $B(G)$ is a free abelian group of finite rank, which coincides with its own centralizer in $G$.
	\end{theorem}
	% \vspace{0.1cm}
	
	\begin{lemma}
	Let $G$ be a just-infinite $p$-group. Then $S(G)=G$, but $B(G)=1$.
	\end{lemma}
	\begin{proof}
	The fact that $G=S(G)$ follows from Proposition \ref{propEquiv} and the fact that finite $p$-groups are nilpotent.
	If $B(G) \neq 1$, then $B(G)$ is a free abelian group by Theorem \ref{thWilson}, which contraddicts that $G$ is a $p$-group.
	\end{proof}
	% \vspace{0.1cm}
	
	\begin{example}[No converse to Proposition \ref{propSubn}] \label{exBad}
	Let $G$ be a just-infinite $p$-group, and let $K \leqslant G$ be any nilpotent subgroup.
	Since every subgroup of finite index of $G$ is subnormal,
	from Theorem \ref{thSubn2} we have that $|HK:H|$ divides $|G:H|$ for every $H \leqslant_f G$.
	On the other hand, $K$ is not subnormal in $G$, because $B(G)=1$.
	\end{example}
	\vspace{0.1cm}
	
	Finally, it is worth to mention the following theorem of Robinson \cite{1970Rob}.
	Given a group property $\mathcal{P}$,
	a group is hyper-$\mathcal{P}$ if every its non-trivial homomorphic image has some non-trivial normal subgroup with the property $\mathcal{P}$.
	
	\begin{theorem}[Theorem 1 in \cite{1970Rob}] \label{thRobinson}
	Let $G$ be a finitely generated hyperabelian or hyperfinite group. If $G$ is an $S$-group, then $G$ is nilpotent.
	\end{theorem}
	% \vspace{0.1cm}
	
	\begin{comment}
	Of course, Proposition \ref{propFitting} implies that $S(G)$ is closed under multiplication if $G$ is finite.
	The same is true in general, but it is quite interesting that we do not know any direct proof of this fact.
	\begin{proposition}
	Let $G$ be any group.
	Then $S(G)$ is a characteristic subgroup of $G$.
	\end{proposition}
	\begin{proof}
	It is obvious that $1 \in S(G)$ and that $S(G)$ is closed by inversion.
	Let $x,y \in S(G)$, and $H \leqslant_f G$.
	Let $N \lhd_f G$ be the normal core of $H$.
	By Lemma \ref{lemReductionRF}, we have $Nx,Ny \in S(G/N)$, and so $Nxy \in S(G/N)$ by Proposition \ref{propFitting}.
	It follows that $N(xy)^{|G:H|}=(Nxy)^{|G:H|} \in H/N$, which implies $(xy)^{|G:H|} \in H$.
	Since $H$ is arbitrary, we obtain $xy \in S(G)$.
	Finally, the fact that $S(G)$ is characteristic follows from its definition.
	\end{proof}
	\vspace{0.1cm}
	\end{comment}
	
	\end{subsection}
	
	\end{section}

	 \vspace{0.2cm}
    \section*{Acknowledgments} 
  The author thanks Bob Guralnick and Orazio Puglisi for useful conversations.

		\vspace{0.5cm}


\vspace{0.2cm}
\thebibliography{10}

 \bibitem{2008Cas} C. Casolo, \textit{Groups with all subgroups subnormal},
	 Note di Matematica \textbf{28 (2)} (2008), 1-153.
	 
	  \bibitem{1993GKL} R. Guralnick, P.B. Kleidman, R. Lyons, \textit{Sylow $p$-subgroups and subnormal subgroups of finite groups},
	 Proceedings of the London Mathematical Society \textbf{66 (1)} (1993), 129-151.
	 
	  \bibitem{1962Kegel} O.H. Kegel, \textit{Sylow-gruppen und subnormaheiler endlicher gruppen},
	 Mathematische Zeitschrift \textbf{78} (1962), 205-221.
	 
	 \bibitem{1991Kleidman} P.B. Kleidman, \textit{A proof of the Kegel–Wielandt conjecture on subnormal subgroups},
	 Annals of Mathematics \textbf{133 (2)} (1991), 369-428.
	
	 \bibitem{2022Levy} D. Levy, \textit{The size of a product of two subgroups and subnormality},
	 Archiv der Mathematik \textbf{118 (4)} (2022), 361-364.
	
	\bibitem{1970Rob} D.J.S. Robinson, \textit{A theorem on finitely generated hyperabelian groups},
	Inventiones Mathematicae \textbf{10} (1970), 38-43.

  \bibitem{1980W} H. Wielandt, \textit{Zusammengesetzte Gruppen: H\"olders Programm heute},
	 Proceedings of Symposia in Pure Mathematics \textbf{37} (1980), 161-173.

\bibitem{1971Wilson} J.S. Wilson, \textit{Groups with every proper quotient finite},
	Mathematical Proceedings of the Cambridge Philosophical Society \textbf{69} (1971), 373-391.
\end{document}